\documentclass[12pt]{amsart}
\usepackage{amsmath,amssymb,amsbsy,amsfonts,amsthm,latexsym,amsopn,amstext,
                            amsxtra,euscript,amscd}

\newfont{\teneufm}{eufm10}
\newfont{\seveneufm}{eufm7}
\newfont{\fiveeufm}{eufm5}
\newfam\eufmfam
                 \textfont\eufmfam=\teneufm \scriptfont\eufmfam=\seveneufm
                 \scriptscriptfont\eufmfam=\fiveeufm
\def\bbbc{{\mathchoice {\setbox0=\hbox{$\displaystyle\rm C$}\hbox{\hbox
to0pt{\kern0.4\wd0\vrule height0.9\ht0\hss}\box0}}
{\setbox0=\hbox{$\textstyle\rm C$}\hbox{\hbox
to0pt{\kern0.4\wd0\vrule height0.9\ht0\hss}\box0}}
{\setbox0=\hbox{$\scriptstyle\rm C$}\hbox{\hbox
to0pt{\kern0.4\wd0\vrule height0.9\ht0\hss}\box0}}
{\setbox0=\hbox{$\scriptscriptstyle\rm C$}\hbox{\hbox
to0pt{\kern0.4\wd0\vrule height0.9\ht0\hss}\box0}}}}
\def\bbbq{{\mathchoice {\setbox0=\hbox{$\displaystyle\rm
Q$}\hbox{\raise
0.15\ht0\hbox to0pt{\kern0.4\wd0\vrule height0.8\ht0\hss}\box0}}
{\setbox0=\hbox{$\textstyle\rm Q$}\hbox{\raise
0.15\ht0\hbox to0pt{\kern0.4\wd0\vrule height0.8\ht0\hss}\box0}}
{\setbox0=\hbox{$\scriptstyle\rm Q$}\hbox{\raise
0.15\ht0\hbox to0pt{\kern0.4\wd0\vrule height0.7\ht0\hss}\box0}}
{\setbox0=\hbox{$\scriptscriptstyle\rm Q$}\hbox{\raise
0.15\ht0\hbox to0pt{\kern0.4\wd0\vrule height0.7\ht0\hss}\box0}}}}
\def\bbbt{{\mathchoice {\setbox0=\hbox{$\displaystyle\rm
T$}\hbox{\hbox to0pt{\kern0.3\wd0\vrule height0.9\ht0\hss}\box0}}
{\setbox0=\hbox{$\textstyle\rm T$}\hbox{\hbox
to0pt{\kern0.3\wd0\vrule height0.9\ht0\hss}\box0}}
{\setbox0=\hbox{$\scriptstyle\rm T$}\hbox{\hbox
to0pt{\kern0.3\wd0\vrule height0.9\ht0\hss}\box0}}
{\setbox0=\hbox{$\scriptscriptstyle\rm T$}\hbox{\hbox
to0pt{\kern0.3\wd0\vrule height0.9\ht0\hss}\box0}}}}
\def\bbbs{{\mathchoice
{\setbox0=\hbox{$\displaystyle     \rm S$}\hbox{\raise0.5\ht0\hbox
to0pt{\kern0.35\wd0\vrule height0.45\ht0\hss}\hbox
to0pt{\kern0.55\wd0\vrule height0.5\ht0\hss}\box0}}
{\setbox0=\hbox{$\textstyle        \rm S$}\hbox{\raise0.5\ht0\hbox
to0pt{\kern0.35\wd0\vrule height0.45\ht0\hss}\hbox
to0pt{\kern0.55\wd0\vrule height0.5\ht0\hss}\box0}}
{\setbox0=\hbox{$\scriptstyle      \rm S$}\hbox{\raise0.5\ht0\hbox
to0pt{\kern0.35\wd0\vrule height0.45\ht0\hss}\raise0.05\ht0\hbox
to0pt{\kern0.5\wd0\vrule height0.45\ht0\hss}\box0}}
{\setbox0=\hbox{$\scriptscriptstyle\rm S$}\hbox{\raise0.5\ht0\hbox
to0pt{\kern0.4\wd0\vrule height0.45\ht0\hss}\raise0.05\ht0\hbox
to0pt{\kern0.55\wd0\vrule height0.45\ht0\hss}\box0}}}}
\def\bbbz{{\mathchoice {\hbox{$\sf\textstyle Z\kern-0.4em Z$}}
{\hbox{$\sf\textstyle Z\kern-0.4em Z$}}
{\hbox{$\sf\scriptstyle Z\kern-0.3em Z$}}
{\hbox{$\sf\scriptscriptstyle Z\kern-0.2em Z$}}}}

 \newtheorem{thm}{Theorem}
 
 \newtheorem{lem}[thm]{Lemma}
 
 \theoremstyle{definition}
 
 \theoremstyle{remark}

\def\cA{{\mathcal A}}

\def\cI{{\mathcal I}}

\def\cP{{\mathcal P}}

\def\cS{{\mathcal S}}

\def\cU{{\mathcal U}}

\def\cX{{\mathcal X}}
\def\cY{{\mathcal Y}}
\def\cZ{{\mathcal Z}}

\def\({\left(}
\def\){\right)}
\def\[{\left[}
\def\]{\right]}
\def\<{\langle}
\def\>{\rangle}

\def\ssum{\mathop{\sum\, \sum}}

\def\fB{{\mathfrak B}}

\def\fl#1{\left\lfloor#1\right\rfloor}
\def\rf#1{\left\lceil#1\right\rceil}

\def\F{\mathbb{F}}

\def\Z{\mathbb{Z}}
\def\Q{\mathbb{Q}}
\def\R{\mathbb{R}}

\def\ep{{\mathbf{\,e}}}

\def\e{{\mathbf{\,e}}}

\def\ord{\operatorname{ord}_p\,}

\def\vec#1{\mathbf{#1}}

\def\mand{\qquad\mbox{and}\qquad}

\begin{document}

\title[Points on   Markoff-Hurwitz Hypersurfaces]{On the Density
of Integer 
Points on  Generalised Markoff-Hurwitz   and Dwork 
Hypersurfaces}

\author
{Mei-Chu Chang}
\address{Department of Mathematics, University of California,
Riverside,  CA 92521, USA}
\email{mcc@math.ucr.edu}

\author{Igor E. Shparlinski} 
\address{Department of Pure Mathematics, University of New South Wales, 
Sydney, NSW 2052, Australia}
\email{igor.shparlinski@unsw.edu.au}

\begin{abstract}  
We use bounds of mixed  character sums modulo a square-free integer $q$
of a special structure 
to estimate the density of integer points on the hypersurface 
$$
f_1(x_1) + \ldots + f_n(x_n) =a x_1^{k_1} \ldots x_n^{k_n} 
$$
for some polynomials $f_i \in \Z[X]$ and nonzero integers $a$
and $k_i$, 
$i=1, \ldots, n$. 
In the case of 
$$
f_1(X) = \ldots = f_n(X) = X^2\mand
k_1 = \ldots = k_n =1
$$ 
the above hypersurface is known as the Markoff-Hurwitz hypersurface, 
while for 
$$
f_1(X) = \ldots = f_n(X) = X^n\mand
k_1 = \ldots = k_n =1
$$ 
it is known as the Dwork hypersurface. 
Our results are
substantially stronger than those known for 
general hypersurfaces. 
\end{abstract}



\maketitle

\section{Introduction}

Studying the density of integer and rational 
points $(x_1, \ldots, x_n)$ on hypersurfaces 
has always been an active 
area of research, where many rather involved methods
have led to remarkable achievements, see~\cite{Brow2,BrHBSa,H-B,H-BP1,Mar1,Mar2,Salb1,Salb2,Tsch} 
and references therein. More precisely, given a  hypersurface 
$$
F(x_1, \ldots, x_n) = 0
$$
defined by a polynomial $F \in \Z[X_1, \ldots, X_n]$ in $n$ variables,
the goal is to estimate the number $N_F(\fB)$ of solutions 
$(x_1, \ldots, x_n) \in \Z^n$ that fall in a hypercube $\fB$ of the 
form 
\begin{equation}
\label{eq:box}
 \fB = [u_1+1,u_1+h]\times \ldots \times [u_n+1, u_n+h]. 
\end{equation}
Unfortunately, even in the most favourable situation, 
the currently known general approaches lead only to a  bound of 
the form $N_F(\fB) = O\(h^{n-2+\varepsilon}\)$ for any fixed 
$\varepsilon>0$ or even weaker, see~\cite{BrHBSa,H-BP1,Salb1,Salb2}.
For some special types of hypersurfaces the strongest known bounds are due 
 Heath-Brown~\cite{H-B} and 
Marmon~\cite{Mar1,Mar2}.
For example, for hypercubes around  the origin, Marmon~\cite{Mar2}
gives a bound of the form  $N_F(\fB) = O\(h^{n-4+\delta_n}\)$ for 
a class of hypersurfaces, 
with some explicit function $\delta_n$ such that $\delta_n  \sim 37/n$ as 
$n \to \infty$.
Combining this bound with some previous results and 
methods, for a certain class of hypersurfaces,  
Marmon~\cite{Mar2} also derives the bound 
$N_F(\fB) = O\(h^{n-4+\delta_n}+ h^{n-3+\varepsilon}\)$
which holds for an arbitrary hypercube $\fB$ with any fixed $\varepsilon>0$ 
and the implied constant that depends only of $\deg F$, $n$ and $\varepsilon$
(note that  $\delta_n> 1$ for $n < 29$).

Finally, we also recall that when the number of variables $n$ is exponentially large 
compared to $d$ and the highest degree form of $F$ is nonsingular, then
the methods developed as the continuation 
of the work of Birch~\cite{Birch} lead to much stronger 
bounds, of essentially optimal order of magnitude.

Here, we show that in some interesting special cases, to which 
further developments of~\cite{Birch} do not apply
(as the highest degree form is singular and the number of variables
is not large enough) 
a modular approach  leads to stronger bounds 
where the saving actually grows with $n$ (at a logarithmic rate).


More precisely we concentrate on hypersurfaces of the form
\begin{equation}
\label{eq:MHD eq}
f_1(x_1) + \ldots + f_n(x_n)  = ax_1^ {k_1} \ldots x_n^{k_n} 
\end{equation}
defined by some polynomials $f_i \in \Z[X]$ and nonzero integers  $a$
and  $k_i$, $i=1, \ldots, n$.
In particular, we use $N_{a, \vec{f}, \vec{k}}(\fB)$
to denote the number of integer solutions to~\eqref{eq:MHD eq}
with $(x_1, \ldots, x_n) \in \fB$, where $\vec{f} = (f_1,\ldots, f_n)$
and $\vec{k} = (k_1,\ldots, k_n)$.

In the case of 
\begin{equation}
\label{eq:MH poly}
f_1(X) = \ldots = f_n(X) = X^2\mand
k_1 = \ldots = k_n =1,
\end{equation}
the equation~\eqref{eq:MHD eq}  defines the {\it Markoff-Hurwitz hypersurface\/},
see~\cite{Bar1,Bar2,Bar3,Cao}, where various questions related to
these   hypersurfaces have been investigated.

Furthermore, for
\begin{equation}
\label{eq:D poly}
f_1(X) = \ldots = f_n(X) = X^n\mand
k_1 = \ldots = k_n =1,
\end{equation}
the equation~\eqref{eq:MHD eq} is known as the {\it Dwork hypersurface\/}, which has been  intensively 
studied by various authors~\cite{Gou,HS-BT,Katz,Kloost,Yu}, in particular, 
as an example of a {\it Calabi--Yau variety\/}.

We remark that solutions with at least one component $x_i=0$, $i=1, \ldots, n$, 
correspond to solutions of a diagonal equation 
$$
\sum_{\substack{j=1\\j \ne i}}^n f_j(x_j)  = -f_i(0)
$$
to which one can apply the standard circle method. 

To clarify our ideas and to make the exposition simpler 
we concentrate here on the solutions to~\eqref{eq:MHD eq}
with $x_1\ldots x_n \ne 0$. In particular, we use 
$N_{a, \vec{f}, \vec{k}}^*(\fB)$ to denote the number 
of such solutions. Clearly for the hypercubes $\fB$ 
of the form~\eqref{eq:box}  we have 
$$
N_{a, \vec{f}, \vec{k}}^*(\fB) = N_{a, \vec{f}, \vec{k}}(\fB). 
$$

Throughout the paper, the implied constants in the symbols ``$O$'', ``$\ll$'' and
``$\gg$'' may depend on 
the  polynomials $\deg f_i$, the coefficient $a$  
and the exponents $k_{i}$ in~\eqref{eq:MHD eq},
$i=1, \ldots, n$, and also on the integer positive 
parameters $r$ and $\nu$.
We recall that
the expressions $A=O(B)$, $A \ll B$ and $B \gg A$ are each equivalent to the
statement that $|A|\le cB$ for some constant $c$.   

Here, we use some ideas from~\cite{Shp}, combined 
a new bound of mixed character sums, that can be of 
independent interest, to derive the following 
result:

\begin{thm}
\label{thm:General1}  
Let 
$f_1(X),\ldots, f_n(X)\in \Z[X]$ be $n$ polynomials
of degrees at most $d$, 
and let $k_1,\ldots, k_n\ge 1$ be odd 
integers. 
For any fixed integer $r\ge 1$, there is a constant $C(r)$
depending only on $r$, such that, 
uniformly over all boxes $\fB$ of the form~\eqref{eq:box}
with 
$$
 \max_{i=1,\ldots, n} |u_i| \le \exp(C(r) h^{4/9})
$$
for the solutions to the equation~\eqref{eq:MHD eq},  we have  
$$
 N_{a, \vec{f}, \vec{k}}^*(\fB) \ll h^{n-4r/9} 
$$
provided that 
$$
n > (d+1)(d+2) 2^{r}\max\left\{2r, 3r-9/2\right\} + 2. 
$$
\end{thm} 

The proof of Theorem~\ref{thm:General1} is based on a bound of 
mixed character sums which combines the ideas from~\cite{Chang1,H-BP2}.

Unfortunately Theorem~\ref{thm:General1} does not apply to the  Dwork hypersurface
as the degrees of the polynomials in~\eqref{eq:D poly} are 
too large for our argument to work. So here apply an alternative
approach that is based on   the 
method of Postnikov~\cite{Post1,Post2} (see also~\cite{Chang2} and
the references therein for further developments). This leads to a 
much more precise 
 bound which however applies only when the degree of the 
 polynomials $f_1(X),\ldots, f_n(X)$ are sufficiently large. 

\begin{thm}
\label{thm:General2}  
Let 
$f_1(X),\ldots, f_n(X)\in \Z[X]$ be $n$ polynomials
of degrees at least $d$, 
and let $k_1,\ldots, k_n\ge 1$ be odd 
integers. 
There is an absolute constant $C$
such that,
uniformly over all boxes $\fB$ of the form~\eqref{eq:box}
with 
$$
 \max_{i=1,\ldots, n} |u_i| \le \exp(C h^{1/3})
$$
and any fixed integer $r\ge 1$ with  
$$
r \le \min_{i=1, \ldots, n} \deg f_i
$$ 
for the solutions to the equation~\eqref{eq:MHD eq}, we have  
$$
 N_{a, \vec{f}, \vec{k}}^*(\fB) \ll h^{n-r/3} 
$$
provided that 
$$
n >  2 r^3 + 1. 
$$
\end{thm} 

Finally, in some cases the arithmetic structure of the right hand side of
the equation~\eqref{eq:MHD eq} allows to derive a much stronger bound 
via the result of~\cite{Chang0}. We illustrate this in the special case of the
equation 
$$
x_1^d + \ldots + x_n^d  = ax_1^ {k_1} \ldots x_n^{k_n} 
$$
and the box $\fB$ aligned along the main diagonal, that is, 
of the form 
\begin{equation}
\label{eq:box diag}
 \fB = [u+1,u+h]\times \ldots \times [u+1, u+h] 
\end{equation}
with some integers $u$ and $h$. 

\begin{thm}
\label{thm:Special}  
Let 
$f_1(X) = \ldots = f_n(X) = X^d$ 
and let $a$, $k_1, \ldots, k_n$ be arbitrary nonzero integers. 
Then, uniformly over all boxes $\fB$ of the form~\eqref{eq:box diag}, 
for the solutions to the equation~\eqref{eq:MHD eq} we have  
$$
 N_{a, \vec{f}, \vec{k}}^*(\fB) \ll h^{d(d+1)/2 + o(1)}. 
$$
\end{thm} 

\section{Some Bounds of Classical Exponential and 
Character Sums} 
\label{sec:charsums gen}

We  denote
$$
\e(z) = \exp(2 \pi i z).
$$

We start with recording the following trivial implication of the orthogonality
of exponential functions.

For quadratic polynomials, we see that~\cite[Theorem~8.1]{IwKow}
implies 

\begin{lem}
\label{lem:LinPoly}
For an integer $q\ge 1$ and any linear polynomial 
$$
G(X) = aX \in \Z[X]
$$ 
 with   $\gcd(a,q)=1$ 
$$
\left|\sum_{z=1}^{H} \e(G(z)/q)\right| \ll  q.
$$
\end{lem}

For quadratic polynomials, we see that~\cite[Theorem~8.1]{IwKow}
yields:

\begin{lem}
\label{lem:QuadrPoly}
For an integer $q\ge 1$ and any quadratic polynomial 
$$
G(X) = aX^2 + bX \in \Z[X]
$$ 
 with   $\gcd(a,q)=1$ 
$$
\left|\sum_{z=1}^{H} \e(G(z)/q)\right| \ll Hq^{-1/2} + q^{1/2}\log q.
$$
\end{lem}

One of our main tools is the following
very special case of a much more general bound
of Wooley~\cite{Wool}, that applies to polynomials with arbitrary real coefficients.

\begin{lem}
\label{lem:Wooley}
For  any polynomial 
$$G(X) = \sum_{i=1}^s \frac{a_i}{q_i} X^i \in \Q[X]
$$ 
of degree $s\ge 3$ with   $\gcd(a_i,q_i)=1$ 
and positive  integer $H$,  for every $j=2, \ldots, s$ we have
$$
\left|\sum_{z=1}^{H} \e(G(z))\right| \ll H\(q_j^{-1} + H^{-1} + q_jH^{-j}\)^\sigma
$$
where 
$$
\sigma = \frac{1}{2(s-1)(s-2)}.
$$
\end{lem}

Let $\cX_q$ be the set of  $\varphi(q)$  multiplicative characters modulo $q$, 
 where $\varphi(q)$ is the Euler function. 
We also denote by  $\cX_q^* = \cX_q \setminus \{\chi_0\}$  the 
set of nonprincipal characters (we set $\chi(0) = 0$ for all $\chi \in \cX_q$).
We appeal to~\cite{IwKow} for a background 
on the basic properties of multiplicative characters and
exponential functions, such as orthogonality. 

We  use the following well-know bound that is implied by 
 the Weil bound for mixed sums of additive and multiplicative characters, 
see~\cite[Chapter~6, Theorem~3]{Li},
and a reduction between complete and incomplete sums, 
see~\cite[Section~12.2]{IwKow}, we also derive the following 
well-known estimate:

\begin{lem}
   \label{lem:WeilIncompl}
For  any $\chi \in \cX_q$, $\lambda \in \F_p$, nonlinear  polynomial $F(X) \in \F_p[X]$
and integers $u$ and $h\ge p$, we have
$$\sum_{x=u+1}^{u+h} \chi(x) \ep(\lambda F(x)) \ll p^{1/2} \log p
$$
provided that $(\chi,\lambda)\ne (\chi_0,0)$. 
\end{lem}

\section{Character  Sums with Square-free Moduli} 
\label{sec:charsums sf}

For a real $Q\ge 3$ and an integer $r\ge1$ we denote by 
$\cP_r(Q)$ the set of integers $q$ of the 
form $q = p_1\ldots p_r$ where  $p_1, \ldots p_r \in [Q,2Q]$ 
are pairwise distinct primes with 
\begin{equation}
\label{eq:gcd pj}
\gcd(k_1\ldots k_n,p_j-1)=1, \qquad j =1, \ldots, r.
\end{equation}

Here we obtain a new bound of mixed character sums with multiplicative 
characters modulo $q \in \cP_r(Q)$ which can be 
of independent interest. We note that recently several bounds
of such sums have been obtained for prime $q=p$, see~\cite{Chang1,H-BP2}.
However for our applications moduli $q \in \cP_r(Q)$  are 
more suitable. Our result is based on the bound of~\cite[Theorem~12.10]{IwKow}
and in fact can be considered as its generalisation.

As  in Section~\ref{sec:charsums gen}, we use   $\cX_q$ for the set of $\varphi(q) = (p_1-1) \ldots (p_r-1)$ multiplicative characters modulo $q = p_1\ldots p_r \in \cP_r(Q)$
and also let $\cX_q^* = \cX_q \setminus \{\chi_0\}$. Furthermore, 
we also continue to use $\e(z) = \exp(2 \pi i z)$. 

We start with recalling the bound of~\cite[Theorem~12.10]{IwKow},
which we present in a somewhat simplified form adjusted to 
our applications. In particular, some simplifications come from 
the fact that the modulus $q \in \cP_s(Q)$ is square-free. 

\begin{lem} 
\label{lem:PureSum} Let $q = \ell_1\ldots \ell_s \in \cP_s(Q)$ for some
primes $\ell_1, \ldots, \ell_s$
and let  $\psi = \chi_1\ldots\chi_s$ be a n multiplicative character of
conductor $q$ and of order $t$, where $\chi_j$ are arbitrary  multiplicative characters of
modulo $\ell_j$, $j=1, \ldots, s-1$, and $\chi_s$ is a nontrivial multiplicative 
character modulo $\ell_s$. Assume   $f(X)$ is a rational function that can be written as
$$
f(X) = \prod_{i=1}^m (X-v_i)^{d_i} 
$$
with some arbitrary integers 
$v_1, \ldots, v_m$  and nonzero integer $d_1, \ldots, d_m$
with
$$
\gcd(d_1, \ldots, d_m, t)=1,
$$
for any integers $u$ and $h$ with $h \ge (2Q)^{9/4}$,  we have 
$$
 \left| \sum_{x=u+1}^{u+h} \psi(f(x))\right| 
\le 4 h \(\gcd(\Delta,\ell_s)  \ell_s^{-1}\)^{2^{-s}} , 
$$
where
$$
\Delta = \prod_{m \ge  i > j \ge 1} (v_{i} - v_{j}).
$$
\end{lem}

We are now ready to present one of our  main technical results
which can be of independent interest.

\begin{lem}
   \label{lem:MixedSum sf}
For any $r=1, 2, \ldots$, a sufficiently large $Q\ge 1$, 
a modulus $q \in \cP_r(Q)$, 
 a polynomial $F(X) \in \R[X]$
of degree $d$ 
and  integers $u$ and $h$ with $h \ge (2Q)^{9/4}$, we have
$$\max_{\chi \in \cX_q^*} 
\left|\sum_{x=u+1}^{u+h} \chi(x) \e(F(x))\right| \ll  h Q^{-\gamma}
$$
where 
$$
\gamma = \frac{1}{2^{r+1}(d+1)(d+2)}.
$$
\end{lem}

\begin{proof} Let us fix some $\chi \in \cX_q^*$. Without loss of 
generality we can write   $\chi = \chi_1\ldots\chi_r$, 
where $\chi_j$ is a multiplicative character modulo  a prime 
$p_j$, $j=1, \ldots, r$ and $\chi_r$ is a nonprincipal character 
(as before, we write $q= p_1\ldots p_r$ for $r$ distinct primes).

Set  $p = p_1$. Then for any positive integer $M$
for the sum
$$
S = \sum_{x=u+1}^{u+h} \chi(x) \e(F(x))
$$
we have 
\begin{equation*}
\begin{split}
S & \le \frac{1}{M} 
\left| \sum_{x=u+1}^{u+h}  \sum_{y=0}^{M-1} \chi(x+py) \e(F(x+py))\right| + 2Mp\\
& \le \frac{1}{M} 
\sum_{\substack{x=u+1\\\gcd(x,p)=1}}^{u+h}
\left|\sum_{y=0}^{M-1} \psi(x+py) \e(F(x+py))\right| + 4MQ,
\end{split}
\end{equation*} 
where $\psi = \chi_2\ldots\chi_r$. 
We note that $\psi$ is of conductor $q/p$ rather that $q$,
so this explains the condition $\gcd(x,p)=1$ in 
the sum over $x$. We can however not simply discard 
this condition and write
\begin{equation}
\label{eq:Prelim}
S   \le \frac{1}{M} 
\sum_{x=u+1}^{u+h} 
\left|\sum_{y=0}^{M-1} \psi(x+py) \e(F(x+py))\right| + 4MQ.
\end{equation}

We divide the unit cube $[0,1]^{d+1}$ into 
$$
K = M^{(d+1)(d+2)/2}
$$ 
cells of the form 
$$
\cU_\vec{a} = \left[\frac{a_0}{M}, \frac{a_0+1}{M}\right] \times \ldots 
\times  \left[\frac{a_d}{M^{d+1}}, \frac{a_d+1}{M^{d+1}}\right], 
$$
where $\vec{a} = (a_0, \ldots, a_{d+1})\in \Z^{d+1}$ runs through
the set $\cA$ of 
integer vectors with  components  $a_\nu = 0, \ldots, M^{\nu+1}-1$,
$\nu =0, \ldots, d+1$.

We now write 
$$
F(X+pY) = F_0(X) + F_1(X) Y + \ldots + F_d(X) Y^d
$$
and define
$$
\Omega_\vec{a} = \{x \in \{u+1, \ldots,u+h\}~:~
(F_0(x), \ldots, F_d(x)) \in \cU_\vec{a}\}, \quad 
\vec{a}\in \cA. 
$$
It is easy to see that for $x \in \Omega_\vec{a}$ we have
$$
\e(F(x+py)) =  E_\vec{a}(y)+ O(M^{-1}), 
$$
where
$$
E_\vec{a}(y) =  \e\(\frac{a_0}{M}+ \frac{a_1}{M^2}y + \ldots +\frac{a_d}{M^{d+1}}y^d\).
$$
Hence we see from~\eqref{eq:Prelim} that 
\begin{equation}
\label{eq:S W}
S \ll  \frac{1}{M} W + h/M + QM, 
\end{equation} 
where 
$$
W = \sum_{\vec{a}\in \cA}
\sum_{x\in \Omega_\vec{a}} \left|\sum_{y=0}^{M-1} \psi(x+py) E_\vec{a}(y)\right|.
$$
We now fix some integer  $k\ge 1$ and apply the H{\"o}lder inequality to $W^{2k}$, getting 
\begin{equation*} 
\begin{split}
W^{2k}& \le \(  \sum_{\vec{a}\in \cA}
\sum_{x\in \Omega_\vec{a}} 1 \)^{2k-1}
 \sum_{\vec{a}\in \cA}
\sum_{x\in \Omega_\vec{a}} \left|\sum_{y=0}^{M-1} \psi(x+py) E_\vec{a}(y)\right|^{2k}\\
 & = h^{2k-1}
 \sum_{\vec{a}\in \cA}
\sum_{x\in \Omega_\vec{a}} \left|\sum_{y=0}^{M-1} \psi(x+py) E_\vec{a}(y)\right|^{2k}.
\end{split}
\end{equation*}
Next, we extend the inner summation over the integers $x\in \Omega_\vec{a}$ to all 
$x \in \{u+1, \ldots, u+h\}$. Opening up the $2k$th power, changing 
the order of summations and using that $|E_\vec{a}(y)| = 1$, we derive
\begin{equation*} 
\begin{split}
W^{2k}& \le   h^{2k-1}  
 \sum_{\vec{a}\in \cA}
\sum_{y_1, \ldots, y_{2k}=0}^{M-1}
\left|\sum_{x=u+1}^{u+h} 
\psi\(\prod_{\nu=1}^k \frac{x+py_\nu}{x+py_{k+\nu}}\) \right|\\
& =  h^{2k-1} K
\sum_{y_1, \ldots, y_{2k}=0}^{M-1}
\left|\sum_{x=u+1}^{u+h} 
\psi\(\prod_{\nu=1}^k \frac{x+py_\nu}{x+py_{k+\nu}}\) \right|.
\end{split}
\end{equation*}

Now, for $O(M^k)$  vectors $(y_1, \ldots, y_{2k})$ 
where each value appears at least twice we estimate the inner 
sum trivially as $h$. 

For the remaining $O(M^{2k})$  vectors 
$(y_1, \ldots, y_{2k})$ we apply Lemma~\ref{lem:PureSum}. 
More precisely, we use it for $s = r-1$ with 
$\ell_i = p_{i+1}$. The rational function $f(X)$ after making all cancellation and
combining equal terms becomes of the form
$$
f(X) = \prod_{i=1}^m \(x+pz_i\)^{d_i},
$$
where $1\le z_1 <  \ldots < z_m \le M$ and at least one $d_i = \pm 1$.
We now assume that 
\begin{equation}
\label{eq:MQ}
M < Q.
\end{equation}
Then we have $\gcd(z_{i} - z_{j}, p_r)=1$ for $m \ge  i > j \ge 1$.
Hence, we also see that 
$$
\gcd\(\prod_{m \ge  i > j \ge 1} (pz_{i} - pz_{j}), p_r\) 
= \gcd\(\prod_{m \ge  i > j \ge 1} (z_{i} - z_{j}), p_r\) = 1.
$$
With the above simplifications, the bound of Lemma~\ref{lem:PureSum} 
becomes 
$$
\left|\sum_{x=u+1}^{u+h} 
\psi\(\prod_{\nu=1}^k \frac{x+py_\nu}{x+py_{k+\nu}}\) \right|
\le  4 h Q^{2^{-r+1}} .
$$
Therefore, 
\begin{equation*} 
\begin{split}
W^{2k} &\ll  h^{2k-1} K\(M^k h + M^{2k} h Q^{2^{-r+1}}\)\\
& =  h^{2k} M^{(d+1)(d+2)/2}\(M^k + M^{2k}  Q^{2^{-r+1}}\) ,
\end{split}
\end{equation*}
which after the substitution in~\eqref{eq:S W} implies
\begin{equation*} 
\begin{split}
S &\ll h M^{(d+1)(d+2)/4k}\(M^{-1/2} + Q^{2^{-r}/k}\)  + h/M + QM\\
& \ll h M^{(d+1)(d+2)/4k}\(M^{-1/2} + Q^{2^{-r}/k}\)  +  h^{8/9}
\end{split}
\end{equation*}
(since by~\eqref{eq:MQ} we have $QM \le Q^2 \ll h^{8/9}$, 
 provided that $h \ge (2Q)^{9/4}$). 
 We now choose $M =  \rf{Q^{2^{-r+1}/k}}$, so~\eqref{eq:MQ} holds, 
 getting 
$$
S \ll h M^{(d+1)(d+2)/4k}  Q^{2^{-r}/k} +  h^{8/9}
= h Q^{((d+1)(d+2)/2k - 1)2^{-r}/k} +  h^{8/9}.
$$
Choosing $k = (d+1)(d+2)$ we conclude the proof. 
\end{proof}

We remark, that the idea of the proof also works with a 
simpler shift $F(x) \to F(x+y)$, however using 
the shift $F(x) \to F(x+py)$ allows to reduce the
conductor (from $q$ to $q/p$) and thus leads to a 
slightly stronger bound as the conductor of $\psi$ is 
now a product of only $r-1$ primes. This idea can be
used in more generality leading to stronger bounds
for more limited ranges of parameters.

We note that we do not impose any conditions on the polynomial 
$F$ in Lemma~\ref{lem:MixedSum sf}, which, in particular can be a constant 
polynomial, in which case, we  have the bound of  of~\cite[Theorem~12.10]{IwKow}.

\section{Character  Sums with Prime-power Moduli} 
\label{sec:charsums pk}

Let $q = p^r$ where $r\ge 1$ is an integer and   $p\ge 3$ is a  prime with 
\begin{equation}
\label{eq:gcd p}
\gcd(k_1\ldots k_n,p-1)=1.
\end{equation}

As  in Section~\ref{sec:charsums gen}, we use   $\cX_q$ for the set of $\varphi(q) = p^{r-1}(p-1)$ multiplicative characters modulo $q$
and let $\cX_q^* = \cX_q \setminus \{\chi_0\}$. We also continue to use $\e(z) = \exp(2 \pi i z)$.

Since group of units modulo $q$ is cyclic then so is $\cX_q$. 
So we now fix a character $\chi \in \cX_q$ that generates 
this group, so that 
$$
\cX = \{\chi^\mu~:~\mu = 0, \ldots, p^{r-1}(p-1) -1 \}. 
$$

The following result is  due to Postnikov~\cite{Post1,Post2},
see also~\cite[Equation~(12.89)]{IwKow}.

\begin{lem} 
\label{lem:Post} Assume that $q = p^r$ for ana integer $r \ge 1$ and  a prime  $p >\max\{2,r\}$. 
Then for any integers $y$ and $z$ with $\gcd(y,p) = 1$, we have
$$
\chi(y+pz) = \chi(y) \e\(F(pwz)/q\)
$$
for some polynomial 
$$
F(Z) = \sum_{k=1}^{r-1} A_k Z^k \in \Z[Z]
$$
of degree $r-1$ and the coefficients satisfying $\gcd(A_k,p) = 1$, $k =1, \ldots, r-1$,
where $w$ is defined by 
$$w y\equiv 1 \pmod q \mand 1 \le w < q.
$$
\end{lem}

\begin{lem}
   \label{lem:MixedSum pk}
Assume that $q = p^r$ for an integer $r \ge1$ and  a prime  $p >\max\{2,r\}$. 
Then for a polynomial $f(X) \in \Z[X]$ of degree $d\ge r$  with the leading 
coefficient $a_d$ satisfying $\gcd(a_d, p) = 1$
and  integers $u$ and $h$ with $q\ge h \ge p^{3}$, uniformly 
over the integers 
$$\lambda \in \{0, \ldots, p^{r}-1\} \mand  \mu \in \{0, \ldots, (p-1) p^{r-1}-1\}
$$ 
with  $\lambda + \mu > 0$, 
we have 
$$
\left|\sum_{x=u+1}^{u+h} \chi^\mu(x) \e(\lambda f(x)/q)\right| \ll  h^{1-1/4r^2}. 
$$
\end{lem}

\begin{proof} Let $H = \fl{h/p}$. 
Then 
\begin{equation}
\label{eq:S H}
\sum_{x=u+1}^{u+h} \chi^\mu(x) \e(\lambda f(x) /q) = S + O(H),
\end{equation}
where 
$$
S = \sum_{y=u+1}^{u+p} \sum_{z=0}^{H}
\chi^\mu(y + pz) \e(\lambda f(y + pz)/q) . 
$$
Therefore, using Lemma~\ref{lem:Post} we obtain 
\begin{equation}
\label{eq:S double}
\begin{split}
S = \sum_{\substack{y=u+1\\\gcd(y,p)=1}}^{u+p} & \chi^\mu(y) \e\(\lambda f(y)/p^r\)\\
& \sum_{z=0}^{H}
 \e\(\sum_{k=1}^{r-1} \frac{1}{p^{r-k}} \(\mu A_ky^{-k} - \lambda f^{(k)}(y)/k!\) z^k\) . 
  \end{split}   
\end{equation}

Let $\ord t$ denote the $p$-adic order of an integer $t$
(where we formally set $\ord 0 = \infty$). 
We set 
$m = \min\{\ord \lambda, \ord \mu\}$. 

In particular, for the inner sum over $z$ in~\eqref{eq:S double}  we have 
\begin{equation}
\label{eq:inner sum}
\begin{split}
 \sum_{z=0}^{H} &
 \e\(\sum_{k=1}^{r-1} \frac{1}{p^{r-k}} \(\mu A_ky^{-k} - \lambda f^{(k)}(y)/k!\) z^k\)\\
& =  \sum_{z=0}^{H}
 \e\(\sum_{k=1}^{r-m-1} \frac{1}{p^{r-m-k}} \(\mu^*A_ky^{-k} - \lambda^*f^{(k)}(y)/k!\) z^k\),
 \end{split}   . 
\end{equation}
where $\mu^* = \mu/p^m$ and $\lambda^* = \lambda/p^m$ are integers. 

We now consider three different cases.

If $m = r-1$ then we see from~\eqref{eq:inner sum} that the inner sum over $z$ in~\eqref{eq:S double}  
is trivial. Note that if $p^{r-1} \mid \mu$ then $\chi^\mu(y)$ becomes  a
character modulo $p$, and it is either a nontrivial character modulo $p$ 
or $\gcd(\lambda^*,p)=1$). 

Thus, using Lemma~\ref{lem:WeilIncompl}, we derive for the sum $S$
\begin{equation}
\begin{split}
\label{eq:S bound1}
S& = H\sum_{\substack{y=u+1\\\gcd(y,p)=1}}^{u+p} \chi^\mu(y) \e\(\lambda^*f(y)/p\) 
\ll Hp^{1/2}\log p \\
& \ll hp^{-1/2} \log h \ll h^{1-1/2r} \log h, 
\end{split}
\end{equation}

If $r-3 \le m  \le r-2$ then we see that the sum~\eqref{eq:inner sum} 
is  a sum with either linear or quadratic polynomial in $z$. 
 Let $\cY$ be the set of solutions 
the congruence 
$$
\mu^*A_{r-m-1}y^{-r+m+1} - \lambda^*f^{(r-m-1)}(y)/(r-m-1)! \equiv 0 \pmod p
$$
where
$$ 
y=u+1, \ldots, u+p, \qquad \gcd(y,p)=1.
$$
Recalling that $\gcd(A_{r-m-1},p)=1$ and  the 
condition on the leading coefficient of $f$, we see that   $\#\cY \le d$. 
Now, for  $y \not \in \cY$, the sum~\eqref{eq:inner sum} is 
\begin{itemize}
\item either a sum with a linear polynomial and a denominator $p$ (when $m=r-2$);
\item or a sum with a quadratic polynomial and a denominator $p^2$ (when $m=r-3$).
\end{itemize}
 Moreover, these polynomials
have the leading coefficient which is relatively prime to $p$. 
In the case of  linear polynomial (that is, $m=r-2$), by Lemma~\ref{lem:LinPoly} 
we bound this sum as $O(p)$.
In the case of a quadratic polynomial  (that is, $m=r-3$),
we bound this sums as $O\(Hp^{-1} + p \log p\)$, which dominates 
the previous bound. Thus, estimating the sum~\eqref{eq:inner sum} 
trivially as $H$ for $y \in \cY$, we derive 
\begin{equation}
\begin{split}
\label{eq:S bound2}
S & \ll H + p\(Hp^{-1} + p \log p\) \ll H + p^2 \log  p\\
& \ll h/p + h^{2/3}\log h \ll h^{1-1/r} \log h.
\end{split}
\end{equation}

Finally, assume that $m \le r-4$. 
For 
$$
j = \rf{\frac{r-m}{2}} \ge 2,
$$
let $\cY$ be the set of solutions  to
the congruence 
$$
\mu^*A_jy^{-j} - \lambda^*f^{(j)}(y)/j! \equiv 0 \pmod p, 
$$
where
$$ 
y=u+1, \ldots, u+p, \qquad \gcd(y,p)=1.
$$
Recalling that $\gcd(A_j,p)=1$ and the 
condition on the leading coefficient of $f$ we see that   $\#\cY \le d$. 
Furthermore, for $y \not \in \cY$, we estimate the inner sum over $z$ by 
Lemma~\ref{lem:Wooley}  with $s= r-m-1 \ge 3$ and $q_j = p^{r-m-j}$, 
getting for  the sum~\eqref{eq:inner sum}:
\begin{equation}
\begin{split}
\label{eq:InnerSum bound1}
 \sum_{z=0}^{H} &
 \e\(\sum_{k=1}^{r-m-1} \frac{1}{p^{r-m-k}} \(\mu^*A_ky^{-k} - \lambda^*f^{(k)}(y)/k!\) z^k\)\\
 &\qquad \qquad\qquad  \ll H (p^{-r+m+j} + H^{-1} +  p^{r-m-j} H^{-j})^\sigma, 
\end{split}
\end{equation}
where
$$
\sigma = \frac{1}{2(r-m-2)(r-m-3)}.
$$
Since $H \ge p^2$ and $j \ge (r-m)/2$ we have 
$$
p^{r-m-j} H^{-j} \le p^{r-m-3j} \le p^{-(r-m)/2}.
$$
On the other hand, since $j \le (r-m+1)/2$, we also have
$$
p^{-r+m+j} \le p^{-(r-m-1)/2}.
$$
Therefore, the bound~\eqref{eq:InnerSum bound1} implies that 
\begin{equation}
\begin{split}
\label{eq:InnerSum bound2}
 \sum_{z=0}^{H} &
 \e\(\sum_{k=1}^{r-m-1} \frac{1}{p^{r-m-k}} \(\mu^*A_ky^{-k} - \lambda^*f^{(k)}(y)/k!\) z^k\)\\
 &\qquad \qquad\qquad\qquad\qquad  \ll H (p^{-(r-m-1)/2} + H^{-1})^\sigma.
\end{split}
\end{equation}
We now note that for $m \le r-4$ we have
$$
\frac{r-m-1}{2} \sigma  = \frac{r-m-1}{4(r-m-2)(r-m-3)} \ge \frac{1}{4r}.
$$
and also 
$$
\frac{2}{3} \sigma  = \frac{1}{3(r-m-2)(r-m-3)} \ge \frac{1}{3r^2}.
$$
Since $p\ge h^{1/r}$ and $H \gg h/p \ge h^{2/3}$, we finally obtain 
\begin{equation}
\begin{split}
\label{eq:InnerSum bound3}
 \sum_{z=0}^{H} 
 \e\(\sum_{k=1}^{r-m-1} \frac{1}{p^{r-m-k}} \(\mu^*A_ky^{-k} - \lambda^*f^{(k)}(y)/k!\) z^k\) & \\
 \ll H& h^{-1/4r^2}.
\end{split}
\end{equation}
So, estimating the sum~\eqref{eq:inner sum} trivially for $y \in \cY$ 
and  using~\eqref{eq:InnerSum bound3} for $y \not \in \cY$, we derive
\begin{equation}
\label{eq:S bound3}
S \ll H +pH h^{-1/4r^2}  \ll h^{1-1/r}  + h^{1-1/4r^2} \ll  h^{1-1/4r^2} 
\end{equation}
Comparing~\eqref{eq:S bound1}, \eqref{eq:S bound2} and~\eqref{eq:S bound3},
we see that the bound~\eqref{eq:S bound3} dominates, and the result follows. 
\end{proof}

\section{Multiplicative Congruences and Equations} 

We make  use of a result of Cochrane and Shi~\cite{CochShi}
that generalises several previous results, which we present in the 
following slightly  less precise form.

\begin{lem}
   \label{lem:4th Moment} For  arbitrary integers
$u$  and $h \le q$, the number of solutions to 
$$
wx \equiv yz \pmod q 
$$
in variables
$$
w,x,y,z\in \{u+1, \ldots, u+h\} \mand \gcd(wxyz,q) = 1,
$$
is bounded by $h^4 q^{-1+o(1)} + h^{2+o(1)}$. 
\end{lem}

Note that in Lemma~\ref{lem:4th Moment} no assumption 
on the modulus $q$ is made (although we apply it only for 
$q\in \cP_r(Q)$).

We  also need a bound of~\cite[Proposition~3]{Chang0} on the number 
of 
divisors in short intervals.

\begin{lem}
\label{lem:Div Int} For any interval $\cI = [u+1, u+h]$ with $h \ge 3$, $u \ge 0$ and
and integer  $z \ge 1$, we have
$$
\# \{(x_1, \ldots, x_n) \in \cI^n~:~ z = x_1 \ldots x_n\} \le \exp\(C_n \frac{\log h}{\log \log h}\)
$$
where $C_n$ is  some absolute constant depending only on $n$. 
\end{lem}

\section{Sets in Reduced Residue Classes}

We need the following simple statement

\begin{lem}
\label{lem:Coprime P} Let $H \ge 3$ be a real number and let $\cS$ be arbitrary set of nonzero integers with 
$|s| \le H$ for $s \in \cS$. For any integer $r \ge 1$ there exists a constant $c(r)$ depending 
only on $r$, such that for any sufficiently large real $Q \ge c(r) \log H$, 
there exists $q \in   \cP_r(Q)$ with  
$$
\#\{s \in  \cS~:~\gcd(s,q) = 1\} 
\ge  \frac{1}{2} \# \cS.
$$
\end{lem}

\begin{proof} 
We have 
\begin{equation*}
\begin{split}
\sum_{q\in \cP_r(Q)}\#\{s \in  \cS~&:~\gcd(s,q)> 1\}\\
&\le \sum_{s\in \cS} \sum_{\substack{q\in \cP_r(Q)\\ \gcd(s,q)> 1}} 1  \le r \sum_{s\in \cS}  \omega(s) \sum_{q\in \cP_{r-1}(Q)} 1,
\end{split}
\end{equation*} 
where as usual, $\omega(s)$ denotes the number of 
prime divisors of $s\ne 0$. We now use that, 
$$
\omega(s) \ll \frac{\log |s|}{\log (2+\log |s|)} \ll \frac{\log H}{\log \log H}
$$ 
(since, trivially $\omega(s)! \le s$)
and also that by the asymptotic formula for the number of  primes in an 
arithmetic progression, we have
$$
\( \frac{Q}{\log Q}\)^{\nu} \ll \# \cP_{\nu}(Q)\ll \( \frac{Q}{\log Q}\)^{\nu}, \qquad \nu =1, 2, \ldots.
$$
Thus, we derive
\begin{equation*}
\begin{split}
\sum_{q\in \cP_r(Q)} \#\{s \in  \cS~:~\gcd(s,q)> 1\} \ll\# \cS
\frac{\log H}{\log \log H}  \( \frac{Q}{\log Q}\)^{r-1}. 
\end{split}
\end{equation*} 
Therefore, 
\begin{equation*}
\begin{split}
\frac{1}{\#\cP_r(Q)} \sum_{q\in \cP_r(Q)} \#\{s \in  \cS~:~\gcd(s,q)> 1\} \ll\# \cS
\frac{\log H}{\log \log H} \cdot \frac{\log Q}{Q} 
\end{split}
\end{equation*} 
and the result now follows. 
\end{proof}

\section{Proof of  Theorem~\ref{thm:General1}}
 
 Take $Q = 0.5 h^{4/9}$. 
By the condition on $\fB$ and 
Lemma~\ref{lem:Coprime P} (applied to the set of all coordinates 
of all $N_{a, \vec{f}, \vec{k}}^*(\fB)$  solutions) there exists $q \in \cP_r(Q)$ such that 
we have 
\begin{equation}
\label{eq:N T}
N_{a, \vec{f}, \vec{k}}^*(\fB) \le 2 T,
\end{equation}
where $T$ is the number of solutions to the congruence
\begin{equation}
\label{eq:cong gen p1pr}
f_1(x_1) + \ldots + f_n(x_n) \equiv a x_1^{k_1} \ldots x_n^{k_n} \pmod q
\end{equation}
with 
$$
(x_1, \ldots, x_n) \in \fB \mand \gcd(x_1\ldots x_n,q) = 1.
$$
Hence it is now sufficient to estimate $T$.

As before, we use   $\cX_q$  to denote the set of multiplicative characters modulo $q$
and also let  $\cX_q^* = \cX_q \setminus \{\chi_0\}$ be the set of 
nonprincipal characters. 

We now proceed as in the proof of~\cite[Theorem~3.2]{Shp}.
Let 
$$
S_i(\chi; \lambda) = \sum_{x=u_i+1}^{u_i+h} \chi^{k_i}(x)
\e\(\lambda f_i(x)/q\), \quad i =1, \ldots, n.
$$
We also introduce the  Gauss sums
$$
G(\chi, \lambda) = \sum_{y=1}^{q} \chi(y)\e(\lambda y/q), \qquad 
\chi \in \cX_q, \ \lambda \in  \Z,
$$

Clearly, we can assume that at least one
of the polynomials $f_1,\ldots, f_n$ is not a constant polynomial
as otherwise the result is  immediate. 

Without loss of generality, we can now assume that $\deg f_1 \ge 1$.
Furthermore, we can also assume that $h$  is sufficiently large  
so that $\gcd(a,q) =1$  and also the leading coefficients 
of the polynomial $f_n$ is relatively prime to $q$ (recall that $q$ is composed out 
of primes in the interval $[Q,2Q]$). 


We now introduce one more variable $y$ that runs through the reduced 
residue system modulo $q$ and rewrite~\eqref{eq:cong gen p1pr} as a system 
of congruences
\begin{equation*}
\begin{split}
&f_1(x_1) + \ldots + f_n(x_n) \equiv y \pmod q ,\\
&a x_1^{k_1} \ldots x_n^{k_n} \equiv y \pmod q.
\end{split}
\end{equation*}
Then exactly as in~\cite[Equation~(3.3)]{Shp}, we write
$$
T = \frac{1}{q\varphi(q)} \sum_{\lambda=1}^q \sum_{\chi\in \cX_q}
 \overline{G}(\chi, \lambda) 
\prod_{i=1}^n |S_i(\chi, \lambda)|,
$$
where, as before,  $\varphi(q)$ is the Euler function 
and $\overline{G}(\chi, \lambda) $ is the complex conjugate of the Gauss sum.

As in the proof of~\cite[Theorem~3.2]{Shp}, we see that, under the 
condition~\eqref{eq:gcd pj}, we have:
\begin{equation}
\label{eq:TR1R2}
T \ll   \frac{1}{q\varphi(q)}\(R_1  + R_2\), 
\end{equation}
where
\begin{equation*}
\begin{split}
R_1 &=    \sum_{\lambda=1}^q \sum_{\chi\in \cX_q^*}
 |G(\chi, \lambda) |
\prod_{i=1}^n |S_i(\chi, \lambda)| ,\\
 R_2 &=   \sum_{\lambda=1}^q 
 |G(\chi_0, \lambda) |
\prod_{i=1}^n |S_i(\chi_0, \lambda)|,
\end{split}
\end{equation*}
 



To estimate $R_1$ we first use Lemma~\ref{lem:MixedSum sf} 
for $n-2$ sums and infer that
\begin{equation}
\label{eq:R1 Prelim}
R_1 \ll  h^{(1-4\gamma/9)(n-2)}  \sum_{\lambda=1}^q \sum_{\chi\in \cX_q^*}
|G(\chi, \lambda)| |S_1(\chi; \lambda)|
|S_2(\chi; \lambda)|,
\end{equation}
where $\gamma$ is as in  Lemma~\ref{lem:MixedSum sf}. 

Using the H{\"o}lder inequality, and then expanding the summation 
to all $\chi \in \cX_q$, we obtain
\begin{equation}
\label{eq:Cauchy}
\begin{split}
 \sum_{\lambda=1}^q \sum_{\chi\in \cX_q} &
|G(\chi, \lambda)| 
|S_1(\chi; \lambda)|
|S_2(\chi; \lambda)|\\
& \le \sum_{\lambda=1}^q  
 \( \sum_{\chi\in \cX_q}
|G(\chi, \lambda)|^2\)^{1/2}\\
& \qquad \qquad 
 \( \sum_{\chi\in \cX_q} |S_1(\chi; \lambda)|^4\)^{1/4}
  \( \sum_{\chi\in \cX_q} |S_2(\chi; \lambda)|^4\)^{1/4}. 
\end{split}
\end{equation}
Using the orthogonality of 
multiplicative characters we see that 
\begin{equation*}
\begin{split}
 \sum_{\chi\in \cX_q} & |S_1(\chi; \lambda)|^4\\
 &= 
q \sum_{\substack{w,x,y,z=u_1+1\\ \gcd(wxyz,q) = 1\\
wx \equiv yz \pmod q}}^{u_1+h} \e\(\frac{\lambda}{q}\(f_1(w) + f_1(x) -f_1(y) -f_1(z)\)\)
 \le  q W, 
\end{split}
\end{equation*}
where $W$ is the number of solutions to 
$$
wx \equiv yz \pmod q 
$$
in variables
$$
w,x,y,z\in \{u_1+1, \ldots, u_1+h\} \mand \gcd(wxyz,q) = 1.
$$

Using Lemma~\ref{lem:4th Moment}, we obtain 
$$
\sum_{\chi\in \cX_q}  |S_1(\chi; \lambda)|^4 \le h^4 q^{o(1)} + h^{2+o(1)}q.
$$
Similarly we obtain the same inequality for the 4th moment of the sums
$S_2(\chi; \lambda)$, and also  
$$
\sum_{\chi\in \cX_q} |G(\chi, \lambda)|^2 \ll q^2.
$$

Thus, collecting these bounds together
which together with~\eqref{eq:R1 Prelim} and~\eqref{eq:Cauchy}, we derive
\begin{equation}
\begin{split}
\label{eq:R1 bound}
R_1 &\ll    h^{(1-4\gamma/9)(n-2)} q^{2} \(h^2 q^{o(1)} + h^{1+o(1)}q^{1/2}\)\\
& =  h^{n-4\gamma(n-2)/9 -1}  \(h q^{2+o(1)} +q^{5/2+o(1)}\). 
\end{split}
\end{equation}

For $R_2$,   using the trivial bound 
$$
|S_i(\chi_0; \lambda)| \le h, \qquad i = 1, \ldots, n-1,
$$ 
we write 
$$
R_2  \le 
h^{n-1} \sum_{\lambda=1}^q 
|G(\chi_0; \lambda)| |S_1(\chi_0; \lambda)|.
$$
We remark that 
$$
G(\chi_0; \lambda) = \sum_{\substack{y=1\\\gcd(y,q)=1}}^{q} \e(\lambda y/q)
$$
is the {\it Ramanujan sum\/} and thus for a square-free $q$ we obtain
$$
|G(\chi_0; \lambda)| = \varphi(\gcd(\lambda,q))
$$
see~\cite[Section~3.2]{IwKow}. Collecting together the values of 
$\lambda$ with the same $\gcd(\lambda,q) = q/s$, where $s$ runs over all
$2^r$ divisors of $q$, and then 
using the Cauchy inequality, we obtain 
\begin{equation*}
\begin{split}
R_2 & \le 
h^{n-1} q \sum_{s\mid q} \frac{1}{s} \sum_{\mu=1}^{s} 
 |S_1(\chi_0; \mu q/s)|\\
&\le 
h^{n-1}q \sum_{s\mid q} \frac{1}{s} \sum_{\mu=1}^{s}  
\left|\sum_{\substack{x=u_1+1\\\gcd(x,q) =1}}^{u_1+h} 
\e\(\mu f_1(x)/s\)\right|\\
&\le 
h^{n-1}q \sum_{s\mid q} \frac{1}{s^{1/2}} \( \sum_{\mu=1}^{s}  
 \left|\sum_{\substack{x=u_1+1\\\gcd(x,q) =1}}^{u_1+h} 
\e\(\mu f_1(x)/s\)\right|^2\)^{1/2}.
\end{split}
\end{equation*}
By the orthogonality of exponential functions, 
$$
\sum_{\mu=1}^{s}  
 \left|\sum_{\substack{x=u_1+1\\\gcd(x,q) =1}}^{u_1+h} 
\e\(\mu f_1(x)/s\)\right|^2 \ \le  s U_s.
$$
Where $U_s$ is the number of solutions to the congruence
$$
f_1(x) \equiv f_1(y) \pmod s, \qquad x,y \in \{u_1+1, \ldots, u_1+h\}. 
$$ 
Since the leading coefficient of $f_1(X)$ is relatively prime to $q$, 
using the Chinese Remainder Theorem we obtain
$$
U_s \ll  h^2/s + h.
$$
Collecting the above inequalities, yields the bound 
\begin{equation}
\label{eq:R2 bound}
R_2  \ll  
h^{n-1}q \sum_{s\mid q} \frac{1}{s^{1/2}} \(h^2 + hs\)^{1/2}
\le h^{n}q. 
\end{equation}

Substituting the bounds~\eqref{eq:R1 bound} and~\eqref{eq:R2 bound} in~\eqref{eq:TR1R2}
and using that 
$\varphi(q) \gg q$ for $q \in \cP_r(Q)$ and also that $q \gg  h^{4r/9}$
we obtain
\begin{equation}
\label{eq:T bound1}
\begin{split}
T & \ll h^{n-4\gamma(n-2)/9 -1}  \(h +q^{1/2}\)q^{o(1)}  + h^{n} q^{-1}\\
& \ll \(h^{n-4\gamma(n-2)/9 }  + h^{n-4\gamma(n-2)/9 -1 + 2r/9}\)q^{o(1)} +   h^{n-4r/9}.
\end{split}
\end{equation}
Clearly, if 
$$
-4\gamma(n-2)/9  < -4r/9
\mand
-4\gamma(n-2)/9 -1 < -2r/3
$$
or, equivalently
$$
n > \max\left\{2^{r+1}(d+1)(d+2)r, 2^{r+1}(d+1)(d+2)(3r/2-9/4)\right\} + 2, 
$$
then the last term dominates in~\eqref{eq:T bound1}.  
Using~\eqref{eq:N T} we conclude the proof.

\section{Proof of  Theorem~\ref{thm:General2}}
 
Take $Q = \fl{0.5 h^{1/3}}$. 
By the condition on $\fB$ and 
Lemma~\ref{lem:Coprime P} (applied to the set of all coordinates 
of all $N_{a, \vec{f}, \vec{k}}^*(\fB)$  solutions and the set $\cP_1(Q)$) there 
exists a prime  $p \in [Q, 2Q]$ such that we have the bound~\eqref{eq:N T}
where now $T$ is the number of solutions to the congruence
\begin{equation}
\label{eq:cong gen pr}
f_1(x_1) + \ldots + f_n(x_n) \equiv a x_1^{k_1} \ldots x_n^{k_n} \pmod {p^r}
\end{equation}
with 
$$
(x_1, \ldots, x_n) \in \fB \mand \gcd(x_1\ldots x_n,p) = 1.
$$
Hence it is now sufficient to estimate $T$. 

As before, we use   $\cX_{p^r}$  to denote the set of multiplicative characters modulo $p^r$
and also let  $\cX_{p^r}^* = \cX_{p^r} \setminus \{\chi_0\}$ be the set of 
nonprincipal characters. 

We now proceed as in the proof of~\cite[Theorem~3.2]{Shp}.
Let 
$$
S_i(\chi; \lambda) = \sum_{x=u_i+1}^{u_i+h} \chi^{k_i}(x)
\e\(\lambda f_i(x)/p^r\), \quad i =1, \ldots, n.
$$
We also introduce the  Gauss sums
$$
G(\chi, \lambda) = \sum_{y=1}^{p^r} \chi(y)\e(\lambda y/p^r), \qquad 
\chi \in \cX_p^r, \ \lambda \in  \Z,
$$

Clearly, we can assume that at least one
of the polynomials $f_1,\ldots, f_n$ is not a constant polynomial
as otherwise the result is  immediate. 

Without loss of generality, we can now assume that $\deg f_1 \ge 1$.
Furthermore, we can also assume that $h$  is sufficiently large  
so that $\gcd(a,p) =1$  and also the leading coefficients 
of the polynomial $f_n$ is relatively prime to $p$ (recall that $p\in [Q,2Q]$). 


We now introduce one more variable $y$ that runs through the reduced 
residue system modulo $q$ and rewrite~\eqref{eq:cong gen pr} as a system 
of congruences
\begin{equation*}
\begin{split}
&f_1(x_1) + \ldots + f_n(x_n) \equiv y \pmod {p^r} ,\\
&a x_1^{k_1} \ldots x_n^{k_n} \equiv y \pmod {p^r}.
\end{split}
\end{equation*}
Then exactly as in~\cite[Equation~(3.3)]{Shp}, we write
$$
T = \frac{1}{p^r\varphi(p^r)} \sum_{\lambda=1}^{p^r} \sum_{\chi\in \cX_{p^r}}
 \overline{G}(\chi, \lambda) 
\prod_{i=1}^n |S_i(\chi, \lambda)|,
$$
where, as before,  $\varphi(q)$ is the Euler function 
and $\overline{G}(\chi, \lambda) $ is the complex conjugate of the Gauss sum.

We see that the contribution from the term corresponding to 
 $\lambda = p^r$ and the principal character $\chi=\chi_0$
is $O(h^n/p^r)$. so 
 the  under the 
condition~\eqref{eq:gcd p}, we have:
\begin{equation}
\label{eq:TR}
T \ll  h^n/p^r + \frac{1}{p^r\varphi(p^r)}R 
\end{equation}
where
$$
R =    \ssum_{\substack{1\le \lambda \le p^r, \ \chi\in \cX_{p^r}\\
(\lambda,\chi)\ne (p^r,\chi_0)} }
 |G(\chi, \lambda) | \prod_{i=1}^n |S_i(\chi, \lambda)| 
$$

To estimate $R$ we first use Lemma~\ref{lem:MixedSum pk} 
for $n-2$ sums and infer that
$$
R \ll  h^{(1-1/4r^2)(n-2)}  \sum_{\lambda=1}^q \sum_{\chi\in \cX_q^*}
|G(\chi, \lambda)| |S_1(\chi; \lambda)|
|S_2(\chi; \lambda)|.
$$
We now proceed exactly as in estimating $R_1$ in the proof of  Theorem~\ref{thm:General1}, getting instead 
of~\eqref{eq:R1 bound} the bound
$$
R \ll   h^{(1-1/4r^2)(n-2)}  p^{2r} \(h^2 p^{o(1)} + h^{1+o(1)}p^{r/2}\).
$$
Since $h^{1/3} \gg p\gg h^{1/3}$ and $r \ge 6$, this simplifies as 
\begin{equation}
\label{eq:R bound}
R \ll   h^{(1-1/4r^2)(n-2)+ 1+o(1)}  p^{5r/2} 
\end{equation}

Substituting the bound~\eqref{eq:R bound}  in~\eqref{eq:TR}, 
we obtain
\begin{equation}
\label{eq:T bound2}
\begin{split}
T & \ll  h^{n - 1 - (n-2)/4r^2 +o(1) } p^{r/2} +  h^n/p^r \\
& \ll    h^{n - 1 - (n-2)/4r^2 +r/6 +o(1)}+  h^{n-r/3}.
\end{split}
\end{equation}

Clearly, if 
$$
r^3 \le \frac{n-2}{2} 
$$
or, equivalently
$$
n \ge  2 r^3 + 2
$$
then the last term dominates in~\eqref{eq:T bound2}.  
Using~\eqref{eq:N T} we conclude the proof.

\section{Proof of  Theorem~\ref{thm:Special}}

Clearly for $(x_1, \ldots, x_n) \in \fB$ where $\fB$ is of the form~\eqref{eq:box diag}
we have 
$$
x_1^d + \ldots + x_n^d  \in \cZ, 
$$
where 
$$
\cZ = \left\{\sum_{\nu =0}^d \binom{d}{\nu} z_\nu u^{d-\nu}~:~ z_\nu \in [0, n h^\nu], \
\nu = 0, \ldots, d\right\}.
$$
In particular, $\# \cZ  \ll h^{d(d+1)/2}$. Applying Lemma~\ref{lem:Div Int} 
to every $z \in \cZ$, we obtain the result.

\section{Comments}
\label{sec:Comm}

We remark that  Theorem~\ref{thm:General1} applies to the
Markoff-Hurwitz hypersurface corresponding to~\eqref{eq:MH poly}.
in which case the condition on $n$ becomes
$$
n > 12 \cdot 2^{r} \max\left\{2r, 3r-9/2\right\} + 2. 
$$

We note that the condition of Theorem~\ref{thm:General1}  
requires $n$ to be only quadratic in $d$, while 
 the saving grows with $n$  as 
$$
\frac{4\log n}{9 \log 2} > 0.64 \log n, 
$$
when $d$ is fixed and $n$ tends to infinity. 

On the other hand,  Theorem~\ref{thm:General1}  
does not apply to the Dwork hypersurface, but 
Theorem~\ref{thm:General2} does and leads to the saving 
that grows with $n$  as 
$$
\frac{(n/2)^{1/3}}{3} > 0.26 n^{1/3}.
$$

It is also easy to see that our methods also works for a more general
form of~\eqref{eq:MHD eq}, namely for the equation 
$$
\(f_1(x_1) + \ldots + f_n(x_n)\)^m  = ax_1^ {k_1} \ldots x_n^{k_n} 
$$
with a nonzero integer $m$. 

One can easily remove the condition on the parity 
of $k_1, \ldots, k_n$ at the cost of essentially only 
typographical changes. Indeed, if some of $k_1, \ldots, k_n$ are even that we 
take all our primes $p$ to satisfy 
$$
p \equiv 3 \pmod {2k_1\ldots k_n}
$$
instead of~\eqref{eq:gcd pj} and~\eqref{eq:gcd p}, and then we deal with
contribution from characters or order $2$ as we have done for the 
principal character. 

Finally, we note that using the bounds of mixed sums from~\cite{H-BP2} within 
our method leads to weaker estimates, but makes them 
fully uniform with respect to the box $\fB$. That is, the conditions 
on $\max_{i=1,\ldots, n} |u_i|$ in Theorems~\ref{thm:General1} 
and~\ref{thm:General2} can be removed at the cost of weakening the
final bound.  

\section*{Acknowledgment}

The authors are  grateful to Oscar Marmon for many useful comments.
The authors also would like to thank Roger Heath-Brown and Lillian  Pierce
for informing them about their work~\cite{H-BP2} when it was still 
in progress and then sending them a preliminary draft.  

During the preparation of this paper, the first author was supported by the 
NSF Grant~DMS~1301608 and the second author by the ARC 
Grant~DP130100237. The second author would also to thank CIRM, Luminy, 
for support and hospitality during his work 
on this project.

\end{document}